\documentclass[finalversion]{pmsart}[2006/01/13]

\usepackage{amsmath,amssymb,amsthm,bm,mathtools,xspace,booktabs,url,tikz,color}

\def\bqny#1{\begin{eqnarray*} #1 \end{eqnarray*}}
\def\bqn#1{\begin{eqnarray} #1 \end{eqnarray}}
\newcommand{\pk}[1]{\mathbb{P} \left( #1 \right) }
\newcommand{\COM}[1]{}

\newcommand{\indi}{\mathbb{I}}
\newcommand{\RL}{\mathbb{R}} 
\newcommand{\nat}{{\mathbb N}} 

\newcommand{\e}{\mathrm{e}} 
\newcommand{\dd}{\mathrm{d}} 
\newcommand{\oh}{{\mathrm{o}}} 
\newcommand{\Oh}{{\mathrm{O}}} 

\newcommand{\Exp}{\mathbb{E}}
\newcommand{\Var}{\mathbb{V}ar}

\newcommand{\Prob}{\mathbb{P}}

\newcommand{\argmin}{\operatornamewithlimits{argmin}}

\newcommand{\halmoss}{{\mbox{\, \vspace{3mm}}} \hfill
\mbox{$\Diamond$}}

\newcommand{\pii}{\pi}
\newcommand\convdistr{
  \xrightarrow{\:\scriptscriptstyle\smash{\mathcal{D}}\:}
}

\newcommand{\Fb}{\overline F}

\newcommand{\cH}{{\mathcal H}}\newcommand{\cHb}{\overline{\mathcal H}}

\newcommand\ginv{
  \tikz[baseline=-0.4ex] \draw[<-] (0,0)--(0.4em,0);
}

\newcommand{\bfy}{\mathbf{y}}
\newcommand{\bfY}{\mathbf{Y}}
\newcommand{\bfX}{\mathbf{X}}
\newcommand{\bfx}{\mathbf{x}}

\renewcommand{\mathsf}{\mathrm} 
\newcommand{\GMDA}{\mathsf{GMDA}}

\renewcommand{\epsilon}{\varepsilon}

\volume{0}%
\fasc{0}%
\years{0000}%
\pages{000--000}%
\firstpage{1} %
\lastrevision{xx.xx.xxxx}

\title{Tail asymptotics of light-tailed Weibull-like sums}
\titlethanks{}
\dedicated{}
\ShortTitle{} 
\ShortAuthors{S{\o}ren Asmussen \etal}
\NumberOfAuthors{4}

\received{}
\FirstAuthor{\mbox{S{\O}ren} Asmussen}
\FirstAuthorCity{Aarhus}
\FirstAuthorAffiliation{Aarhus University}
\FirstAuthorAddress{Department of Mathematics, Aarhus University,
					Ny Munkegade, DK-8000 Aarhus C, Denmark}
\FirstAuthorEmail{asmus@imf.au.dk}                              
\FirstAuthorThanks{}                             

\SecondAuthor{\mbox{Enkelejd} Hashorva} 
\SecondAuthorCity{Lausanne} %
\SecondAuthorAffiliation{University of Lausanne}%
\SecondAuthorAddress{Department of Actuarial Science, University of Lausanne, UNIL-Dorigny, 1015 Lausanne, Switzerland}%
\SecondAuthorEmail{Enkelejd.Hashorva@unil.ch}%
\SecondAuthorThanks{}%

\ThirdAuthor{\\ \mbox{Patrick J.} Laub}
\ThirdAuthorCity{Brisbane}%
\ThirdAuthorAffiliation{University of Queensland \& Aarhus University}%
\ThirdAuthorAddress{Department of Mathematics, The University of Queensland, Brisbane, Queensland 4072, Australia}%
\ThirdAuthorEmail{p.laub@uq.edu.au}
\ThirdAuthorThanks{}%

\FourthAuthor{\mbox{Thomas} Taimre}
\FourthAuthorCity{Brisbane}%
\FourthAuthorAffiliation{University of Queensland}%
\FourthAuthorAddress{Department of Mathematics, The University of Queensland, Brisbane, Queensland 4072, Australia}%
\FourthAuthorEmail{t.taimre@uq.edu.au}%
\FourthAuthorThanks{}%

\FifthAuthor{} 
\FifthAuthorCity{} %
\FifthAuthorAffiliation{}%
\FifthAuthorAddress{} %
\FifthAuthorEmail{}%
\FifthAuthorThanks{}%

\keywords{Gumbel MDA; Laplace transform; Rare event simulation; Sums of of random variables; Weibull distribution}                                     

\MSCcodes{Primary: 62E20, 60F05; Secondary: 60E15, 65C05, 60F10.}               

\begin{document}

\maketitle

\begin{abstract}
We consider sums of $n$ i.i.d.\ random variables with tails close to $\exp\{-x^\beta\}$ for some $\beta>1$. Asymptotics  developed
by Rootz\'en (1987) and Balkema, Kl\"uppelberg \& Resnick (1993) are discussed from the point of view of tails rather of densities, using a somewhat different angle,  and supplemented with bounds, results on a random number $N$ of terms, and simulation algorithms.
\end{abstract}

\section{Introduction}\label{S:Intr}

Let $X,X_1,\ldots,X_n$ be i.i.d.\ with common distribution $F$. A recurrent theme in applied probability is then
to determine the order of magnitude of the tail $\Prob(S_n>x)$ of their sum $S_n=X_1+\cdots+X_n$.

The results vary according to the heaviness  of the tail $\Fb=1-F$ of $F$. In the heavy-tailed case, defined as the $X$ for which $\Exp\e^{sX}=\infty$ for all $s>0$, there is the subexponential class in which the results take a clean form (see e.g.\ \cite{EKM} or \cite{RP}). In fact, by the
very definition of subexponentiality, we have $\Prob(S_n>x)$ $\sim$ $n\Fb(x)$ as $x\to\infty$ where $\Fb(x)=\Prob(X>x)$. The main examples are regularly varying $\Fb(x)$, lognormal $X$, and Weibull tails $\Fb(x)=\e^{-cx^\beta}$ where $0<\beta<1$.

In the light-tailed case, defined as the $X$ for which $\Exp\e^{sX}<\infty$ for some $s>0$, the most standard asymptotic regime is not $x\to\infty$ but rather
$x=x_n$ going to $\infty$ at rate $n$. For example, let $x_n=nz$ for some $z$, where typically $z>\Exp X$ in order to make the problem a rare-event one.
Under some regularity conditions, the sharp asymptotics are then given by the saddlepoint approximation $\Prob(S_n>x)$ $\sim$
$c(z)\e^{-nI(z)}/n^{1/2}$ for suitable $c(z)$ and $I(z)$, cf.\ \cite{JLJ}. This is a large deviations result, describing how likely
it is for $S_n$ to be far from the value $n\Exp X$ predicted by the LLN. However, in many applications the focus is rather on a small
or moderate $n$, i.e.\ the study of  $\Prob(S_n>x)$ as $x\to\infty$ with $n$ fixed.

The basic light-tailed explicit examples in this setting are the exponential distribution, the gamma distribution,
the inverse Gaussian distribution,  and the normal distribution. The tail of $F$ is exponential or close-to-exponential for exponential, gamma and inverse Gaussian distributions; this is the borderline between light and heavy tails, and the analysis of tail behaviour is relatively simple in this case (we give a short summary
later in Section~\ref{C1C2}).
The most standard class of distributions with a lighter tail is formed by the Weibull distributions where
$\Fb(x)=\e^{-cx^\beta}$ for some $\beta>1$. For $\beta=2$, this is close to the normal distribution, where (by its well-known Mill's ratio) $\Fb(x)\sim \e^{-x^2/2}/(\sqrt{2\pii}x)$ when $F = \Phi$ is the standard normal law. The earliest study of  tail properties of $S_n$ may be that
of \cite{Rootzen} which was later followed up by the mathematically deeper and somewhat general study of Balkema, Kl\"uppelberg, \& Resnick \cite{BKR}, henceforth referred to as BKR. The setting of both papers is densities.

Despite filling an obvious place in the theory of  tails of sums, it has been our impression that this theory is less known
than it should be. This was confirmed by a Google Scholar search which gave only 27 citations of BKR,
most of which were even rather peripheral.
One reason may be that the title \emph{Densities with Gaussian tails} of BKR is easily misinterpreted,
another the heavy analytic flavour of the paper. Also note that the focus of~\cite{Rootzen} is somewhat different and
the set of results we are interested in here appears as a by-product at the end of that paper.

The purpose of the present paper is twofold: to present a survey from a somewhat different angle than BKR, in the 
hope of somewhat remedying this situation; and to supplement the theory with various new results.
In the survey part, the aim has been  simplicity and intuition more than generality.
In particular, we avoid considering convex conjugates and some non-standard
central limit theory developed in Section~6 of BKR. These tools are mathematically deep and elegant, but
not really indispensable for developing what we see as the main part of the theory.
Beyond this expository aspect, our contributions are: to present the main results and their conditions
in terms of  tails rather than densities;  to develop simple upper and lower
bounds; to study the case of a random
number of terms $N$, more precisely properties of $\Prob(S_N>x)$
when $N$ is an independent Poisson r.v.; and to look into simulation
aspects.

The precise assumptions on the distribution $F$ in the paper vary somewhat depending on the context and progression of
the paper. The range goes from the vanilla Weibull tail $\Fb(x)=\e^{-cx^\beta}$ via an added power in the asymptotics,
$\Fb(x)\sim dx^{\alpha}\e^{-cx^\beta}$, to the full generality of the BKR set-up. Here $cx^\beta$ is replaced by a smooth convex
function $\psi(x)$ satisfying $\psi'(x)\to\infty$ and the density has the form $\gamma(x)\e^{-\psi(x)}$ for a function $\gamma$ which
is in some sense much less variable than $\psi$ (the precise regularity conditions are given in Section~\ref{S:BKR}).

\section{Heuristics}\label{S:Heur}

With heavy tails, the basic intuition on the tail behaviour of $S_n$ is the principle of
a single big jump; this states that a large value of $S_n$ is typically caused by
one summand being large while the rest take ordinary values.
A rigorous formulation of this can be proved in a few lines from the very definition
of subexponentiality, see e.g.~\cite[p.\,294]{RP}. With light tails, the folklore is that if
$S_n$ is large, say $S_n\approx x$, then all $X_i$ are of the same order $x/n$.

This suggest that the asymptotics of $\Prob(S_n>x)$ are essentially determined by the
form of $F$ locally around $x/n$. A common type of such local behaviour is
that $\Fb\bigl(x+e(x)y\bigr)$ $\sim$ $ \Fb(x)\e^{-y}$ for some positive  function $e(x)$
as $x\to\infty$ with $y\in\RL$ fixed; this is abbreviated as $F \in \GMDA(e)$. Equivalently,
\begin{equation}\label{17.12a}
\Lambda\bigl(x+e(x)y\bigr)\ \sim\ \Lambda(x) + y
\end{equation}
where $\Lambda(x)=-\log\Fb(x)$.
Here one can take $e(x)=\Exp[X - x \mid X > x]$, the so-called \emph{mean excess function}; if $F$ admits a density $f(x)$,   an alternative asymptotically equivalent choice is the \emph{inverse hazard rate} $e(x) = 1/\lambda(x)$ where $\lambda(x)=\Lambda'(x)=f(x)/\Fb(x)$.

In fact, \eqref{17.12a} is a necessary and sufficient condition for $F$
to be in $\GMDA(e)$, the maximum domain of attraction of the Gumbel distribution \cite{EKM}. Even if this condition may look
special at first sight, it covers the vast majority of well-behaved light-tailed distributions, with
some exceptions such as certain discrete distributions like the geometric or Poisson.

From these remarks one may proceed for $n=2$
from the convolution,
\begin{align} \label{Convolution}
	\nonumber \MoveEqLeft \Prob(X_1 + X_2 >x) = (f \ast \Fb)(x)
	= \int_{-\infty}^\infty\lambda(z) \exp\bigl\{-\Lambda(z)-\Lambda(x-z)\bigr\}\,\dd z \\
	&= \int_{-\infty}^\infty
	\frac{ e(x/2) }{ e\Big(x/2 + e(x/2)y \Big)  } \exp \Big\{ {-} \Lambda \Big(\frac{x}{2} + e\Big(\frac{x}{2}\Big) y \Big)-\Lambda \Big(\frac{x}{2}- e\Big(\frac{x}{2}\Big) y \Big) \Big\} \,\dd y,
\end{align} 
where we have substituted $z=x/2+e(x/2) y$. First note that if $\lambda(x)$ tends to 0 as $x\to \infty$ and is differentiable, we can expand $\Lambda$ about $y=0$ as
\begin{align*}
	\Lambda \Big(\frac{x}{2} + e\Big(\frac{x}{2}\Big) y \Big) &\sim \Lambda \left(\frac{x}{2}\right) + y +\frac{\lambda '\left(\frac{x}{2}\right)}{2 \lambda \left(\frac{x}{2}\right)^2} \,y^2 \,.
\end{align*}
By defining $\sigma^2(u)=\lambda(u)^2/2\lambda'(u)$ and repeating this argument we get that
\begin{align} \label{LambdaAsym}
	\Lambda \Big(\frac{x}{2} \pm e\Big(\frac{x}{2}\Big) y \Big) &\sim \Lambda \left(\frac{x}{2}\right) \pm y +\frac{y^2}{4 \sigma^2(\frac{x}{2})}.
\end{align}
Also we will use that $e(x)$ is self-neglecting, i.e.\ $\forall\, t$, $e\bigl(x+e(x)t\bigr)$ $\sim$ $e(x)$ as $x\to\infty$, as is well-known and easy to prove from \eqref{17.12a}.
Combining \eqref{LambdaAsym} and the self-neglecting property with \eqref{Convolution} gives us
\begin{align}\nonumber \MoveEqLeft
\Prob(X_1+X_2>x) \sim \int_{-\infty}^\infty 1 \cdot \exp \Big\{ {-}2 \Lambda \left(\frac{x}{2}\right) - \frac{y^2}{2 \sigma^2(\frac{x}{2})} \Big\} \,\dd y \\
\label{16.11c}&=\ \sqrt{2\pii\sigma^2(x/2)} \, \exp\bigl\{-2\Lambda(x/2)\bigr\} \,.
\end{align}
In summary, rewriting \eqref{16.11c} gives
\begin{equation}\label{17.12b}
\overline{F^{*2}}(x)\ =\ \Prob(X_1+X_2>x)\ \sim\ \Fb(x/2)^2\sqrt{\pii \frac{\lambda(x/2)^2}{\lambda'(x/2)}}\,.
\end{equation}
The key issue in making this precise is to keep better track of the second order term in the Taylor
expansion, as discussed later in the paper.

\begin{commentaryRemark}\label{Rem:17.12a}
The procedure to arrive at \eqref{17.12b} is close to the Laplace method for obtaining integral
asymptotics. Classically, the integral in question has the form $\int_a^b\e^{-\theta h(z)}\,\dd z$
and one proceeds by finding the $z_0$ at which $h(z)$ is minimum and performing a
second order Taylor expansion around $z_0$. Here, we neglected the   $\lambda(z)$  in front
and took the relevant analogue  of $z_0$ as $x/2$ which is precisely the minimizer
of $\Lambda(x-z)+\Lambda(z)$.
\halmoss\end{commentaryRemark}

\begin{commentaryRemark}\label{Rem:17.12b} If $X_1,X_2$ have different distributions $F_1,F_2$, the
above calculations suggest that $X_1+X_2 > x$ will occur roughly when $X_1\approx z(x)$,
$X_2\approx x-z(x)$ where $z=z(x)$ is the solution of $\lambda_1(z)=\lambda_2(x-z)$.
In fact, this is what is needed to make the first order Taylor terms cancel. For example, if $\Fb_1(x)=\e^{-x^{\beta_1}}$, $\Fb_2(x)=\e^{-x^{\beta_2}}$
with $\beta_2<\beta_1$,
 we get $z(x)\sim cx^\eta$ where $\eta=(\beta_2-1)/(\beta_1-1)<1$,
 $c=(\beta_2/\beta_1)^{1/(\beta_1-1)}$. This type of heuristic is an
 important guideline when designing importance sampling algorithms,
 cf.~\cite[V.1,\,VI.2]{SSAA}.
\halmoss\end{commentaryRemark}

\section{Weibull-like sums}\label{S:Sums}

We now make the heuristics of preceding section rigorous for the case of
different distributions $F_1,F_2$ of $X_1,X_2$ such
that the densities $f_1,f_2$ satisfy
\begin{equation}\label{SA4.2a} f_i(x)\ \sim d_ix^{\alpha_i+\beta-1}\e^{-c_i x^\beta}\,,\quad x\to\infty,\ i=1,2
\end{equation}
for some common $\beta>1$, where the $\alpha_i$ can take any value in $(-\infty,\infty)$ and $c_i,d_i$ are positive ($i=1,2$).

We start by some analytic preliminaries. Given \eqref{SA4.2a}, we define
\begin{equation}\label{SA4.2aa}
\eta=c_1^{1/(\beta-1)}+c_2^{1/(\beta-1)},\ \
\theta_1=c_2^{1/(\beta-1)}/\eta,\ \ \theta_2=c_1^{1/(\beta-1)}/\eta,\ \ 
\kappa=\frac{\eta^{\beta-1}}{\beta c_1c_2}\,.
\end{equation}
Note that
\begin{equation}\label{9.11d}
\displaystyle \Fb_i(x)\,\sim\,\frac{d_i}{\beta c_i}x^{\alpha_i}\e^{-c_i x^\beta}
\end{equation}
(hence $c_i=1$, $d_i=\beta$, $\alpha_i=0$ corresponds to the traditional Weibull tail $\e^{-x^\beta}$).
Define the excess function of $F_i$ by $e_i(x)=\Fb_i(x)/f_i(x)$. Thus
$e_i(x)$ is the inverse hazard rate and has asymptotics $x^{1-\beta}/(\beta c_i) $ with limit 0 as $x\to\infty$.

\begin{lemma}\label{Lemma:9.11a} Define
$c\,=\, c_1\theta_1^\beta+c_2\theta_2^\beta$.
Then $c<\min(c_1,c_2)$, $\theta_1+\theta_2=1$, and
\begin{equation}\label{11.11b}
e_1(\theta_1 x) \sim e_2(\theta_2 x) \sim \frac{\kappa}{x^{\beta-1}} = \frac{1}{\beta c_1\theta_1^{\beta-1}x^{\beta-1}} = \frac{1}{\beta c_2\theta_2^{\beta-1}x^{\beta-1}}.
\end{equation}
\end{lemma}
\begin{proof} All statements are obvious except $c<\min(c_1,c_2)$. But
\begin{align}\label{SA4.2b}
c\ &=\ c_1\theta_1^{\beta-1}\theta_1+c_2\theta_2^{\beta-1}\theta_2
\ =\ \frac{c_1c_2\theta_1}{\eta^{\beta-1}}
+\frac{c_1c_2\theta_2}{\eta^{\beta-1}}\ =\ \frac{c_1c_2}{\eta^{\beta-1}}\\ 
\nonumber &<\ \frac{c_1c_2}{{[c_2^{1/(\beta-1)}]}^{\beta-1}}
\ =\ c_1.
\end{align}
Similarly, $c<c_2$.
\end{proof}

\begin{lemma}\label{Lemma:10.11a} $\displaystyle  (1+h)^\beta = 1+h\beta +\frac{h^2}{2}
\beta(\beta-1)\omega(h)$
where $\omega(h)\to 1$ as $h\to 0$ and
$\underline \omega_\epsilon=\inf_{-1+\epsilon<h<\epsilon^{-1}}\omega(h)>0$ for all $\epsilon>0$.
\end{lemma}
\begin{proof}  By standard Taylor expansion results, $\omega(h)=(1+h^*)^{\beta-2}$.
where $h^*$ is between $0$ and $h$. The statement on $\underline \omega_\epsilon$ follows from this by considering
all four combinations of the cases $h\le 0$ or $h>0$, $1<\beta\le 2$ or $\beta\ge 2$ separately.
\end{proof}

The key result is the following. It allows, for example,  
to determine the asymptotics of the tail
or density of $F^{*n}$ in the Weibull-like class by a straightforward 
induction argument, see Corollary~\ref{SACor4.2} below.
\begin{theorem}\label{Th:9.11a} Under assumption \eqref{SA4.2a},
$\displaystyle \Prob(X_1+X_2>x)\ \sim\ k x^{\gamma}\e^{-cx^\beta}$ as $x\to\infty$, where
$ \gamma=\alpha_1+\alpha_2+\beta/2$ and
$k\,=\, d_1 d_2 \theta_1^{\alpha_1} \theta_2^{\alpha_2} \kappa \eta^{1-\beta} (2\pii \sigma^2)^{1/2} / \beta$,
with $\theta_1,\theta_2,\kappa,\eta$ as in \eqref{SA4.2a}, the constant $c$ as in Lemma~\ref{Lemma:9.11a}, and $\sigma^2$ determined
by 
\[\frac{1}{\sigma^2}=\frac{1}{\sigma_1^2}+\frac{1}{\sigma_2^2}\quad\text{where }\frac{1}{\sigma_i^2}\ =\ \beta(\beta-1)c_i\theta_i^{\beta-2}\kappa^2 \,. \]
Further the density of $X_1+X_2$ has asymptotic form
$\beta ck x^{\gamma+\beta-1}\e^{-cx^\beta}$.
\end{theorem}
\begin{commentaryRemark}\label{Rem:9.11a}\rm If $F_1=F_2$ and $c_1=c_2=1$, then $\theta_1=\theta_2=1/2$ and
$c=1/2^{\beta-1}$ in accordance with Section~\ref{S:Heur}.
\halmoss\end{commentaryRemark}
\begin{proof}
By Lemma~\ref{Lemma:9.11a}, we can choose $0<a_-<a_+<1$ such that $a_+^\beta c_2>c$, $(1-a_-) c_1>c$.
Then
\[\Prob\bigl (X_1+X_2>x,X_1\not \in [a_-x,a_+x]\bigr)\ \le \Prob(X_1>a_+x)+ 
\Prob(X_2>(1-a_-)x)
\]
is $\oh(x^{\gamma}\e^{-cx^\beta})$ and so it suffices to show that
\begin{equation}\label{9.11a}\Prob\bigl (X_1+X_2>x, a_-x<X_1<a_+x\bigr)
\ =\ \int_{a_-x}^{a_+x} f_1(z)\Fb_2(x-z)\,\dd z\end{equation}
has the claimed asymptotics. The last expression together with $a_->0$, $a_+<1$ also
shows that the asymptotics is a tail property so that w.l.o.g.\ we may assume that
$e_i(\theta_i x)=\kappa/x^{\beta-1}$,  implying  that \eqref{11.11b} holds with equality.

Now
\begin{align}\nonumber \MoveEqLeft  \Prob\bigl (X_1+X_2>x, a_-<X_1<a_+x\bigr)\ =\
\int_{a_-x}^{a_+x} f_1(z)\Fb_2(x-z)\,\dd z\\ \label{11.11a} &= \ \int_{a_-x}^{a_+x} \frac{d_1d_2}{\beta c_2}
z^{\alpha_1+\beta-1}(x-z)^{\alpha_2}
\exp\bigl\{-c_1z^\beta-c_2(x-z)^\beta\bigr\}\,\dd z.
\end{align}
Using the substitution $z=\theta_1x+y\kappa/x^{\beta-1}$, we have $x-z= \theta_2x-y\kappa/x^{\beta-1}$,
\begin{align}\label{9.11b}
c_1z^\beta+c_2(x-z)^\beta \ =\ c_1\theta_1^\beta x^\beta
\bigl(1+h_1(x,y)\bigl)^\beta+c_2\theta_2^\beta x^\beta
\bigl(1-h_2(x,y)\bigl)^\beta
\end{align}
where $h_i(x,y)=y\kappa/\theta_ix^\beta$.
Taylor expanding  $\bigl(1\pm h_i(x,y)\bigl)^\beta$ as in Lemma~\ref{Lemma:10.11a}
and using \eqref{11.11b}, the first order term of \eqref{9.11b} is
\[ c_1\theta_1^\beta x^\beta +c_2\theta_2^\beta x^\beta + \beta c_1\theta_1^{\beta-1}\kappa - \beta c_2\theta_2^{\beta-1}\kappa =c x^\beta. \]
Defining
$\omega_1(x,y)=\omega\bigl(h_1(x,y)\bigr)$, $\omega_2(x,y)=\omega\bigl(-h_2(x,y)\bigr)$, \eqref{11.11a} becomes
 \begin{multline}\nonumber
\frac{d_1d_2}{\beta c_2} \int_{y_-(x)}^{y_+(x)}
\bigl(\theta_1x+e_1(\theta_1x)y\bigr)^{\alpha_1+\beta-1}\bigl(\theta_2x-e_2(\theta_2x)y\bigr)^{\alpha_2}
\\ \cdot\ \exp\Bigl\{-cx^\beta
-\frac{y^2}{2\sigma_1^2x^\beta}\omega_1(x,y)-\frac{y^2}{2\sigma_2^2x^\beta}\omega_2(x,y)
\Bigr\}\,\frac{\kappa}{x^{\beta-1}}\dd y
\end{multline}
where $y_-(x)=(a_--\theta_1)x^\beta/\kappa$, $y_+(x)=(a_+-\theta_1)x/e(\theta_1 x)$.
Notice here that $a_-x<z<a_+x$ ensures the bound
\[ h_1(x,y)=\frac{1}{\theta_1x}(z-\theta_1x)\ge
\frac{a_-}{\theta_1}-1\ > \ -1.
\]
Similarly, $-h_2(x,y)\ge -a_+/\theta_2-1>1$. Using Lemmas~\ref{Lemma:9.11a} and \ref{Lemma:10.11a} shows that
so that the $\omega_i(x,y)$ are uniformly bounded below, and that
$\bigl(\theta_ix+e_i(\theta_ix)y\bigr)/x$ is bounded in $y_-(x)<y<y_+(x)$ and goes to $\theta_i$ as $x\to\infty$. A dominated convergence argument gives therefore that the asymptotics of \eqref{9.11a} is the same as that of
\begin{align*}& \frac{d_1d_2\kappa }{\beta c_2}\theta_1^{\alpha_1+\beta-1}\theta_2^{\alpha_2}x^{\alpha_1+\alpha_2}\e^{-cx^\beta}
 \int_{-\infty}^\infty \exp\Bigl\{-\frac{y^2}{2\sigma^2x^\beta}\Bigr\}\,\dd y\\ =\ &
 \frac{d_1d_2\kappa\eta^{1-\beta} }{\beta }\theta_1^{\alpha_1}\theta_2^{\alpha_2}x^{\alpha_1+\alpha_2}\e^{-cx^\beta}(2\pii \sigma^2 x^{\beta})^{1/2}\ =\
  k x^{\gamma} \e^{-cx^\beta} \,.
\end{align*}
This proves the assertion on the tail of $X_1+X_2$, and the proof of the density claim differs only by constants.
\end{proof}

\begin{corollary} \label{SACor4.2}
Assume the density $f$ of $F$ satisfies $f(x)\,\sim\, dx^{\alpha+\beta-1}\e^{-c x^\beta}$ as $x\to\infty$. Then the tail and the density of an i.i.d.\ sum satisfy
\begin{align}\label{SA4.2c}
\overline{F^{*n}}(x)\ =\ \Prob(S_n>x)\ &\sim \ k(n) x^{\alpha(n)}\e^{-c(n) x^\beta}\,,
\\ \label{SA4.2d}  f^{*n}(x)\ &\sim\ \beta c(n) k(n) x^{\alpha(n)+\beta-1}\e^{-c(n) x^\beta}
\end{align}
where $c(n)=c/n^{\beta-1}$, $\alpha(n) = n\alpha+(n-1)\beta/2$ and
\begin{align}\label{SA4.2g}
k(n) = \frac{d^n}{\beta c} \Big[ \frac{2\pi}{\beta(\beta-1) c} \Big]^{(n-1)/2} n^{\frac{1}{2} (\beta -n (2 \alpha +\beta )-1)}.
\end{align}
\end{corollary}
\begin{proof}
We use induction. The statement is trivial for $n=1$ so assume it proved for $n-1$. 
Taking $F_1=F$, $F_2=F^{*(n-1)}$ and applying Theorem~\ref{Th:9.11a} implies the result, and provides
recurrences for $c(n)$, $\alpha(n)$, and $k(n)$. To be specific, say that the $F_i$  distributions have densities $f_i$ like
\[ f_i(x) \sim d_i(n) x^{\alpha_i(n)+\beta-1} \e^{-c_i(n) x^\beta} \,, \quad i=1,2. \]
As $F_1=F$ is fixed, we simply have $c_1(n) = c$, $d_1(n) = d$, $\alpha_1(n) = \alpha$, and for $F_2=F^{*(n-1)}$ the induction hypothesis gives us
\[ c_2(n) = \frac{c}{(n-1)^{\beta-1}}, \quad
 \quad d_2(n) = \beta c_2(n-1) k(n-1), \quad \alpha_2(n) = \alpha(n-1) . \]

We extend the notation of Theorem~\ref{Th:9.11a} in the obvious way, for example we define $\eta(n)= c_1(n)^{1/(\beta-1)} + c_2(n)^{1/(\beta-1)}$. These simplify to
\[ \eta(n) = \frac{n c^{1/(\beta-1)}}{n-1}, \quad \theta_1(n) = \frac{1}{n}, \quad \theta_2(n) = \frac{n-1}{n}, \quad \kappa(n) = \frac{n^{\beta-1}}{\beta c} \,. \]
So $c(n) = c_1(n) \theta_1(n)^\beta + c_2(n) \theta_2(n)^\beta = c/n^\beta + c (n-1)/n^\beta = c/n^{\beta-1}$. Also, we have
$\alpha(n) = \alpha_1(n) + \alpha_2(n) + \beta/2  = n \alpha + (n-1) \beta / 2$. 

The last recursion is less simple. We need the $\sigma$ constants:
\[ \sigma_1^2(n) = \frac{\beta c n^{-\beta}}{\beta-1}, \quad \sigma_2^2(n) = \frac{\beta c (n-1) n^{-\beta}}{\beta-1} ,\quad \sigma^2(n) = \frac{\beta c (n-1) n^{-\beta-1}}{\beta-1} . \]
Setting $k(1) = d/(\beta c)$, we get for $n\ge 2$
\begin{align*}
k(n) &= d_1(n) d_2(n) \theta_1(n)^{\alpha_1(n)} \theta_2(n)^{\alpha_2(n)} \kappa(n) \eta(n)^{1-\beta} (2\pii \sigma(n)^2)^{1/2} / \beta \\
&= \Big[ \frac{2\pi}{\beta(\beta-1) c} \Big]^{1/2} d (n-1)^{\alpha  (n-1)+\frac{1}{2} (\beta  (n-2)+1)} n^{-\alpha n-\frac{1}{2} \beta  (n-1) -\frac{1}{2}} k(n-1) \\
&= \frac{d^n}{\beta c} \Big[ \frac{2\pi}{\beta(\beta-1) c} \Big]^{(n-1)/2} \prod_{\ell=2}^{n} (\ell-1)^{\alpha  (\ell-1)+\frac{1}{2} (\beta  (\ell-2)+1)} \ell^{-\alpha \ell-\frac{1}{2} \beta  (\ell-1) -\frac{1}{2}} \\
&= \frac{d^n}{\beta c} \Big[ \frac{2\pi}{\beta(\beta-1) c} \Big]^{(n-1)/2} n^{\frac{1}{2} (\beta -n (2 \alpha +\beta )-1)}.
\end{align*}

\end{proof}

Note that \eqref{SA4.2d} is already in Rootz\'en~\cite{Rootzen} (see his equations 
(6.1)--(6.2)). We point out later that the assumptions on the density can be relaxed
to $\Fb(x)\sim kx^{\alpha}\e^{-c x^\beta}$ where $k=d/c\beta$.

\section{Light-tailed sums}\label{S:BKR}

We now proceed to the set-up of BKR and first
introduce some terminology related to the densities of the form  $f(x)\sim \gamma(x)\e^{-\psi(x)}$.  The main assumption is that the function $\psi$ is non-negative, convex,
$C^2$,  and its first order derivative is denoted $\lambda$. Further it is supposed that
\bqn{\label{lambdaI}
	\lim_{x\to \infty}\lambda(x)=\infty,
}
$\lambda'$ is ultimately positive and $1/\sqrt{\lambda'}$ is self-neglecting, i.e.\ that for $x\to\infty$
\begin{equation}\label{17.12c}
\lambda'\bigl(x+y/\sqrt{\lambda'(x)}\bigr)\ \sim\ \lambda'(x).
\end{equation}
A function $\gamma$ is called \emph{flat} for $\psi$ if locally uniformly on bounded $y$-intervals
\begin{equation}\label{17.12d}
\lim_{x\to \infty}\frac{ \gamma\bigl(x+y/\sqrt{\lambda'(x)}\bigr)}{\gamma(x)}=1.
\end{equation}
Similar conventions apply to functions denoted $\psi_1,\psi_2,$ etc.
For the Weibull case,
\[ \psi(x)=ax^\beta,\ \ \lambda(x)=a\beta x^{\beta-1},\ \ \gamma(x) = \lambda(x) \]
and so \eqref{17.12c} and \eqref{17.12d} are satisfied.
Examples beyond Weibull-like distributions are $\psi(x)=x\log x$ and $\psi(x)=\e^{ax},a>0$. 

Define the class $\cH(\gamma,\psi)$ as the class of all distributions $F$ having a density
of the form $\gamma(x)\e^{-\psi(x)}$ where $\psi$ is as above and $\gamma$ a measurable function which is flat for $\psi$,
and let  $\cHb(\gamma,\psi)$ be the class of  distributions $F$ satisfying 
$\Fb(x)$ $\sim$ $\gamma(x)\e^{-\psi(x)}/\lambda(x)$.

\begin{theorem}\label{Th:19.12a}
{\rm (i)} $\cH(\gamma,\psi)\,\subseteq\,\cHb(\gamma,\psi)$;\\
 {\rm (ii)} Assume $F_1\in \cH(\gamma_1,\psi_1)$, $F_2\in \cH(\gamma_2,\psi_2)$. Then
 $F_1*F_2\in\cH(\gamma,\psi)$,
 where $\gamma,\psi$ are determined by first solving
 \begin{equation}\label{19.12a} q_1+q_2=x,\quad \lambda_1(q_1)=\lambda_2(q_2) \end{equation}
 for $q_1=q_1(x)$, $q_2=q_2(x)$ and next letting $\psi(x)=\psi_1(q_1)+\psi_2(q_2)$,
 \[\gamma(x)\ =\ \sqrt{ \frac{2\pi \lambda'(x)}{\lambda_1'(q_1) \lambda_2'(q_2)}}
\gamma_1(q_1)\gamma_2(q_2)
 \]
 where $\lambda(x)=\psi'(x)=\lambda_1(q_1)=\lambda_2(q_2)$.\\
  {\rm (iii)} Assume $F_1\in \cHb(\gamma_1,\psi_1)$, $F_2\in \cHb(\gamma_2,\psi_2)$. Then
	there exists $H_i \in  \cH(\gamma_i,\psi_i), H_i \in \GMDA(1/\lambda_i)$ and
	$$\overline H_i(x) \sim \overline F_i(x), \quad \overline{H_1*H_2}(x) \sim
	\overline{F_1*F_2}(x)  .$$
Moreover,   $F_1*F_2\in\cHb(\gamma,\psi)$ with $\gamma,\psi$ as in {\rm (ii)} and
$F_1*F_2\in \GMDA(1/\lambda)$.
\end{theorem}

The proof of Theorem~\ref{Th:19.12a} is in Appendix~\ref{S:ProofofTheorem}.
Part (ii) is in BKR, here slightly reformulated, and
a number of examples in BKR can be obtained as corollaries of this theorem.

\begin{commentaryRemark}\label{Rem:19.12a}
Letting $\tau(y)=\lambda_1^{{\ginv}}(y)+\lambda_2^{{\ginv}}(y)$, the solution of \eqref{19.12a} can be written
\begin{equation}\label{19.12aa} q_1(x)\,=\,\lambda_1^{{\ginv}}\bigl(\tau^{{\ginv}}(x)\bigr) \,,\quad
q_2(x)\,=\,\lambda_2^{{\ginv}}\bigl(\tau^{{\ginv}}(x)\bigr)\end{equation}
(here $\cdot^{{\ginv}}$ means functional inverse).
\halmoss\end{commentaryRemark}

\section{Bounds}\label{S:Bounds}

There are easy upper- and lower-tail bounds for Weibull sums in terms of
the incomplete gamma function $\Gamma(\alpha,x)=$
$\int_x^\infty t^{\alpha-1} \e^{-t}\,\dd t$ when $\beta> 1$ that in their simplest form just come from thinking
about $p$-norms  $\|\bfy\|_p= $ $\bigl(|y_1|^p+\cdots+|y_n|^p\bigr)^{1/p}$ and the  fact that if $Y$ is standard exponential,
then $Y^{1/\beta }$ is Weibull with tail $\e^{-x^\beta}$.

\begin{proposition}\label{Prop:24.1a} Let $X$ have density $\beta k^{\gamma/\beta} x^{\gamma-1}\e^{-kx^\beta}
/\Gamma(\gamma/\beta)$, $x>0$, where $k>0$, $\beta\geq 1$, and $\gamma>0$. Then
\[ \frac{\Gamma(n\gamma/\beta,kx^\beta)}{\Gamma(n\gamma/\beta)}
\ \le\ \Prob(X_1 + \cdots + X_n > x)\ \le \
\frac{\Gamma(n\gamma/\beta,kx^\beta/n^{\beta-1})}{\Gamma(n\gamma/\beta)}\,.\]
\end{proposition}
\begin{proof}
An $X$ with the given density has the same distribution as $(Y/k)^{1/\beta}$
where $Y$ is $\mathsf{Gamma}(\alpha,1)$ with density
$y^{\alpha-1}\e^{-y}/\Gamma(\alpha)$, where $\alpha=\gamma/\beta$.
Therefore
\[ X_1^\beta + \cdots + X_n^\beta\ =\ \|\bfX\|_\beta^\beta \ \stackrel{d}{=} \ \|\bfY/k\|_1=
Y_1/k+ \cdots + Y_n/k\,,\]
where $Y_1,\ldots,Y_n$ are i.i.d.\ $\mathsf{Gamma}(\alpha,1)$.
From the Jensen and H\"older inequalities we have for $p\ge 1$ and $\bfx \in \mathbb{R}^n$ that
$$ \|\bfx\|_p \le  \|\bfx\|_1 \le  \|\bfx\|_p  n^{1- 1/p}. $$
Hence, since further $ \|\bfY\|_1=Y_1+\cdots+Y_n$ is $\mathsf{Gamma}(n\alpha,1)$ with tail $\Gamma(n\alpha,y)/
\Gamma(n\alpha)$, one has for any $x>0$
\bqny{\Prob(X_1 + \cdots + X_n > x)\ &=&\ \Prob\bigl( \|\bfX\|_{1}>x\bigr)\\
	&\le&   \Prob\bigl( \|\bfX\|_\beta^\beta >x^\beta/n^{\beta-1}\bigr)=
	\Prob\bigl( \|\bfY\|_1>kx^\beta/n^{\beta-1}\bigr)\,,
}
and similarly for the lower bound.
\end{proof}
The (upper) incomplete gamma function $\Gamma(\alpha,x)$ appearing here is available in most
standard software, but note that an even simpler lower bound comes from
$\Gamma(\alpha,x)\ge x^{\alpha-1}\e^{-x}$ for $x>0$ when $\alpha = \gamma/\beta \geq 1$, resp.\
$\Gamma(\alpha,x)\ge x^{\alpha-1}\e^{-x} \times \big( x/(x+1-\alpha) \big)$
when $\alpha \in (0,1)$.
Moreover, observe that $X$ with the density given in Prop.~\ref{Prop:24.1a} has tail probability
\[
\overline{F}_X(x) = \Prob(X>x) = \frac{\Gamma(\gamma/\beta,k x^\beta)}{\Gamma(\gamma/\beta)}\,.
\]
Hence, appealing to the fact that $\Gamma(\alpha,x) \sim x^{\alpha-1} \e^{-x}$ as $x\to\infty$,
the upper bound in Prop.~\ref{Prop:24.1a} is asymptotically
\[
\frac{\Gamma(\gamma/\beta)^n}{\Gamma(n\gamma/\beta)}\,n^{n\gamma/\beta -1} \, k^{n-1}\,\left(\frac{x}{n}\right)^{\beta(n-1)} \overline{F}_X(x/n)^n\,.
\]

When $\gamma=\beta$ (the ordinary Weibull case), the ratio of this upper bound to the
true asymptotic form for $\Prob(X_1 + \cdots + X_n > x)$ is
\[
\frac{n^{(n-1/2)}}{(n-1)!}\,\Bigl[ \frac{(\beta-1)}{2\pi\beta}\Bigr]^{(n-1)/2} \,k^{(n-1)/2}\,\left(\frac{x}{n}\right)^{\beta(n-1)/2} \,,
\]
so the upper bound is out only by a polynomial factor in $x$, which indicates
it is close to the true probability on a logarithmic scale.
More precisely, writing $U(x)$ for the upper bound and $P(x)$ for the true probability,
it holds trivially that $x^{-1} \, \log(U(x)) \sim x^{-1} \, \log(P(x))$ as $x\to\infty$.

It is straightforward to extend Prop.~\ref{Prop:24.1a} to the following slightly more general form.
\begin{proposition}\label{Prop:24.1b} Let $\{X_i\}_{i=1}^n$ be independent random variables with
density $\beta k^{\gamma_i/\beta} x^{\gamma_i-1}\e^{-kx^\beta}
/\Gamma(\gamma_i/\beta)$, $x>0$, where $k>0$, $\beta\geq 1$, and $\gamma_i>0$, for $i=1,\dots,n$. Then with $\gamma_0=\sum_{i=1}^n \gamma_i$, it holds that
\[ \frac{\Gamma(\gamma_0/\beta,kx^\beta)}{\Gamma(\gamma_0/\beta)}
\ \le\ \Prob(X_1 + \cdots + X_n > x)\ \le \
\frac{\Gamma(\gamma_0/\beta,kx^\beta/n^{\beta-1})}{\Gamma(
\gamma_0/\beta)}\,.\]
\end{proposition}

\section{M.g.f.'s and the exponential family}\label{S:ExpFam}

In this section, we assume that $X \sim F$ has the tail asymptotics $\gamma(x)\e^{-x^\beta}/\lambda(x)$ for some $\beta>1$ where $\lambda(x)=\beta x^{\beta-1}$.
Define
\begin{align*}
\widehat F[\theta]\ &=\ \Exp [\e^{\theta X}]\ =\
\int_{-\infty}^\infty \e^{\theta z}\, F(\dd z)\,,\quad
F_\theta(\dd z)\ =\ \frac{\e^{\theta z}}{\widehat F[\theta]}F(\dd z)
\end{align*}
where expectations with respect to $F_\theta$ will be denoted $\Exp_\theta[\cdot]$.
Determining the asymptotics of $\widehat F[\theta]$ and characteristics of the exponential family
like their moments is easier when taking $\theta=\lambda(x)$. For a general $\theta$, one then
just have to substitute $x=\lambda^{\ginv}(\theta)$ in the following result.

\begin{proposition}\label{Prop:18.12a} As $x\to\infty$, it holds that
	\begin{align}\label{18.12a} \widehat  F\bigl[\lambda(x)\bigr]\ &\sim\ \sqrt{\frac{2\pii}{\lambda'(x)}}\gamma(x)\e^{(\beta-1)x^\beta},\\
	\label{18.12b}\Exp_{\lambda(x)}X\ &\sim\ x.
	\end{align}
	Further, we have the following convergence in $\Prob_{\lambda(x)}$-distribution as $x\to \infty$ 
	\begin{align} \label{18.13}
	\sqrt{\lambda'(x)}(X-x)=\sqrt{\beta(\beta-1) x^{\beta-2}}(X-x)\,\convdistr\,N\bigl(0,1\bigr) \,.
	\end{align}
\end{proposition}
\begin{proof}
	Suppose for simplicity that $X$ is non-negative.
	In view of Proposition 3.2 in BKR we can assume w.l.o.g. that $\gamma \in C^\infty$. 
	By Theorem \ref{Th:19.12a} we have that $\overline{F}(x) \sim \overline{H}(x)$ where $H$ has the density $\gamma(z) \e^{-z^\beta}$ for $z \ge 0$.
	It follows easily from our proof below that $\Exp[X^k \e^{\lambda(x) X}] \sim \Exp[X^k_* \e^{\lambda(x) X_*}]$ for $k \ge 0$ with $X_* \sim H$, so we assume w.l.o.g. that $F$ has the density $f(z)= \gamma(z) \e^{-z^\beta}$ for $z \ge 0$. 
	Lastly, we have that $\gamma(x)= \oh(\e^{cx})$ for any $c>0$, as
	\bqn{
	\label{2.12i}
	\lim_{x\to \infty}\frac{\gamma'(x)}{\sqrt{\lambda'(x)} \gamma(x)}\ &=\ 0.}

	For $g(z)= z^k \e^{\lambda(x) z}$, with $k\ge 0$, it follows by integration by parts that
	\bqn{
		\Exp[X^k \e^{\lambda(x) X}] &=& g(0) + \int_{0}^\infty g'(z) \overline{ F}(z) \dd z \notag \\
		&=& \indi\{k=0\} + \int_{0} ^{c_1x} g'(z) \overline{F}(z) \dd z + \int_{c_1x}^\infty g'(z) \overline{F}(z) \dd z \notag \\
		&=& \Oh(\e^{ \tilde c_1 x ^\beta}) + \int_{c_1x}^\infty \big[ k z^{k-1} + \lambda(x) z^k \big] \e^{\lambda(x)z} \frac{\gamma(z)}{\lambda(z)} \e^{-z^\beta} \dd z \label{int_by_parts}
	}
	for any $0< c_1< \tilde c_1< 1$ sufficiently small. 

	Consider integrals of the form $\int_{c_1x}^\infty z^k \e^{\lambda(x)z} \gamma(z) \e^{-z^\beta} \dd z$ and note that the global maximum of the exponent $\lambda(x)z- z^\beta$ is at $z=x$. 
	We use the substitution, similar to those in Sections~\ref{S:Heur} and \ref{S:Sums}, of $z=x+ y/\lambda(x)$ and note that
	\[\lambda(x) z-z^\beta\ \sim \ (\beta-1)x^\beta-\frac{y^2\lambda'(x)}{2\lambda(x)^2}.\]
	Therefore, for any $D > 0$ we have for $x \to \infty$
	\bqny{ 
		&& \int_{c_1x}^\infty z^k \frac{\gamma(z)}{\lambda(z)} \e^{\lambda(x) z- z^\beta} \dd z \
	  	\sim\ \int_{x- D/\lambda(x)}^{x+ D/\lambda(x)} z^k \frac{\gamma(z)}{\lambda(z)} \e^{\lambda(x) z- z^\beta} \dd z \\
		&\sim & \int_{-D}^D  \bigl(x+ \frac{y}{\lambda(x)}\bigr)^k
		\frac{\gamma\big(x+ \frac{y}{\lambda(x)}\bigr)}{\lambda\big(x+ \frac{y}{\lambda(x)}\bigr)} 
		\exp\Bigr\{ (\beta-1)x^\beta-\frac{y^2\lambda'(x)}{2\lambda(x)^2} \Bigr\}\,\frac{1}{\lambda(x)}  \dd y\\
		&\sim & x^k \frac{\gamma(x)}{\lambda(x)^2}\e^{(\beta-1)x^\beta}\int_{-D}^D
		\exp\Bigr\{ {-}\frac{y^2\lambda'(x)}{2\lambda(x)^2} \Bigr\}\,\dd y\
		\sim \ \sqrt{\frac{2 \pii}{\lambda'(x)}} x^k \frac{\gamma(x)}{\lambda(x)} \e^{(\beta-1)x^\beta}
	}
	where the replacement of the limits $\pm\, D$ by $\pm\,\infty$ follows from  $\lambda'(x)/\lambda(x)^2\to 0$.
	Combining this integral asymptotic with \eqref{int_by_parts}  we get 
	\begin{align}
	\Exp[X^k \e^{\lambda(x) X}]\ 
	&= \Oh(\e^{ \tilde c_1 x ^\beta}) 
	+ k \int_{c_1x}^\infty z^{k-1} \frac{\gamma(z)}{\lambda(z)} \e^{\lambda(x)z - z^\beta}  \dd z
	\\ &\qquad + \lambda(x) \int_{c_1x}^\infty z^k \frac{\gamma(z)}{\lambda(z)} \e^{\lambda(x)z - z^\beta}  \dd z \notag \\
	&=\ \Oh(\e^{ \tilde c_1 x ^\beta}) + 
	\sqrt{\frac{2 \pii}{\lambda'(x)}} \gamma(x) \e^{(\beta-1)x^\beta} \Big( x^k + \frac{k}{\lambda(x)} x^{k-1}  \Big), \label{derivs_of_mgf}
	\end{align}
	or to take only the largest term,
	\[
	 	\Exp[X^k \e^{\lambda(x) X}] \sim 
	 	\sqrt{\frac{2\pii}{\lambda'(x)} } \gamma(x) x^k \e^{(\beta-1)x^\beta} \quad \text{as } x \to \infty.
	 \]
	 From this \eqref{18.12a}--\eqref{18.12b} are easy.
	
	Next, we show the asymptotic normality. By the above arguments, we assume for simplicity  that $F$ has density $f(z)=\gamma(z) \e^{-z^\beta} $ for all $z>0$. Similarly,
	writing instead $z=x+y/\sqrt{\lambda'(x)}$, we have
	\[\lambda(x) z-z^\beta\ \sim \ (\beta-1)x^\beta-\frac{y^2}{2} .\]
	For some $D< \min(0,v)$ we obtain
	\bqny{
		\lefteqn{\int_{x+ D/\sqrt{\lambda'(x)}}^{x+ v/\sqrt{\lambda'(x)}}
			\gamma(z)\exp\Bigr\{ \lambda(x) z- z^\beta \Bigl\} \dd z}\\
		&\sim&  \frac{1}{\sqrt{\lambda'(x)}} \int_{D}^v
		\gamma\big(x+ y/\sqrt{\lambda'(x)}\bigr)\exp\Bigr\{ (\beta-1)x^\beta-\frac{y^2}{2} \Bigr\}\, \dd y\\
		&\sim&  \frac{1}{\sqrt{\lambda'(x)}} \gamma(x) \e^{(\beta-1)x^\beta}\int_{D}^{v }
		\exp\Bigr\{ {-}\frac{y^2}{2} \Bigr\}\,\dd y.
	}Hence, letting $D\to - \infty$ yields
	$$ 	\Exp[\e^{\lambda(x)X};\, \sqrt{\lambda'(x)}(X-x)\le v]
	\sim \sqrt{\frac{2\pii}{\lambda'(x)}} \gamma(x) \e^{(\beta-1)x^\beta}\Phi(v).
	$$
	Dividing by \eqref{18.12a} gives $\Prob_{\lambda(x)}\bigl(\sqrt{\lambda'(x)}(X-x)\le v\bigr)$ $\to$ $\Phi(v)$
	which is \eqref{18.13}.
\end{proof}

\begin{commentaryRemark}\label{Rem:Lo}
	Asymptotic normality for the general case $\overline{F}(x)= \e^{-\psi(x)}$ similar to the result of Proposition \ref{Prop:18.12a} is derived in \cite{BKR03}.
\halmoss\end{commentaryRemark}

\begin{commentaryRemark}\label{Rem:21.12a} The BKR method of proof is modelled after the standard
	proof of the saddlepoint approximation: exponential change of measure using estimates of the above type.
	One has
	\begin{equation}\label{21.12b}
	\Prob(S_n>x)\ \ =\ \widehat{F}[\theta]^n\Exp_\theta\bigl[\e^{-\theta S_n};\, S_n>x\bigr]
	\end{equation}
	and should take $\theta$ such that $\Exp_\theta S_n=x$, i.e.\ $\theta=\lambda(x/n)$.
	The approximately normality of $(X_1,\ldots,X_n)$ gives that $S_n$ is approximately normal
	$\bigl(x,n/\lambda'(x/n)\bigr)$. So, one can compute
	\[\Exp_{\lambda(x/n)}\exp\bigl\{-a\sqrt{\lambda'(x/n)/n}\,S_n\bigr\}\]
	for any fixed $a$
	but $\theta=\lambda(x/n)$ is of a different order than $\sqrt{\lambda'(x/n)/n}$. Therefore
	(as for the saddlepoint approximation) a sharper CLT is needed, and
	this is maybe the most demanding part of the BKR approach.	\halmoss\end{commentaryRemark}\

\section{Compound Poisson sums}\label{S:CompP}

 We  consider here $S_N=X_1+\cdots+X_N$ where
$N$ is Poisson$(\mu)$ and independent of $X_1,X_2,\ldots$, where $X_i \sim \mathsf{Weibull}(\beta)$. The asymptotics of
$\Prob(S_N>x)$ are important in many applications, for example actuarial sciences \cite{RP}, and can be investigated using classical saddlepoint techniques. The relevant asymptotic is the classical Esscher approximation :
\begin{align}
	\Prob(S_N > x) 
	&\sim \frac{
		\big( \widehat{F_{S_N}}[\theta] - \e^{-\mu} \big) \exp \{ -\theta x \}
	}{
		\theta \sigma_c(\theta)
	} B_0(\ell)  \label{jlj_first}  \,,
\end{align}
where $\theta$ is the solution to $\mu \widehat{F}'[\theta] = x$, and $\widehat{F_{S_N}}[\theta] = \exp \{ \mu (\widehat{F}[\theta] - 1 ) \}$, 
$B_0(l) = l \e^{l^2/2}(1 - \Phi(l)) \to (2\pi)^{-1/2}$, $\sigma_c^2(\theta)= \mu F''[\theta]$, and $\ell = \theta \sigma_c(\theta)$. See (7.1.10) in~\cite{JLJ},
where also further
refinements and variants are given. The issue with implementing \eqref{jlj_first} is that we do not usually have access to $\widehat{F}[\theta]$; note, Mathematica can derive $\widehat{F}[\theta]$ when $\beta =\ $1.5, 2, or 3.

For standard $\mathsf{Weibull}(\beta)$ variables, \eqref{18.12a} simplifies to
\begin{align*} 
\widehat{F}[t] \sim \sqrt{\frac{2 \pi \beta ^{\frac{1}{1-\beta }}}{\beta -1}}  t^{\frac{\beta}{2(\beta-1)}} \e^{(\beta -1) (t/\beta)^{\frac{\beta }{\beta -1}}} =: \widetilde{F}[t]\,.
\end{align*}
Unfortunately $\widehat{F_{S_N}}[t] \not\sim \exp\{\mu( \widetilde{F}[t] - 1)\}$, though $\widehat{F_{S_N}}[t] \approx_{\log} \exp\{\mu( \widetilde{F}[t] - 1)\}$, where the notation $h_1(x)\approx_{\log}h_2(x)$ means that $\log h_1(x)/\log h_2(x)\to 1$.

One can select the $\theta$ which solves $\mu \widetilde{F}'[\theta]=x$, however it seems this must be done numerically.
An alternative is the asymptotic forms for $\widehat{F}^{(k)}$ from \eqref{derivs_of_mgf}. Take
\begin{align} \label{deriv_asymptotics}
	\widehat{F}^{(k)}[\theta] = \Exp[X^k\e^{\theta X}]\sim y^k\widehat{F}[\theta], \quad \text{for } k \in \nat
\end{align}
where we've written $\theta=\lambda(y)$ as in Section~\ref{S:ExpFam}. Thus if we set $\theta$ as the solution to $\mu y \widetilde{F}[\lambda(y)] = x$ then we get
\begin{equation} \label{y_approx}
y = 2^{-1/\beta } \Bigg[ \frac{(\beta +2) }{(\beta -1) \beta } \mathcal{W}\Bigg( \frac{(\beta -1) \beta }{(\beta +2)} \left(\frac{2^{\frac{1}{\beta }+\frac{1}{2}} x}{c_1}\right)^{\frac{2 \beta }{\beta +2}} \Bigg) \Bigg]^{1/\beta } 
\end{equation}
where $\mathcal{W} $ is the Lambert W function and $c_1 = \mu \sqrt{2 \pi } \beta   / \sqrt{(\beta -1) \beta }$.

With this choice of $\theta$, we can say $\widehat{F}^{(k)}[\theta] \sim x y^{k-1}$, so $\sigma_c^2(\theta) \sim \mu x y$ and $\ell \sim \lambda(y) \sqrt{\mu x y}$, and substituting this into \eqref{jlj_first} gives us
\begin{equation} \label{asym}
\Prob(S_N > x) 
	\approx_{\log} \frac{
		\e^{-\mu} \big( \exp \{ \mu x / y \} - 1 \big) \exp \{ -\theta x \}
	}{
		\lambda(y) \sqrt{\mu x y}
	} B_0(\ell) \,.
\end{equation}

Preliminary numerical work indicates that \eqref{asym} is not particular accurate in the whole range of relevant parameters. 
The problem derives from the fact we only have log-asymptotics for $\widehat{F_{S_N}}[\theta]$; finding more accurate asymptotics is left for future work.

A further interesting extension could be the asymptotic form of $\Prob\bigl(Z(t)>x\bigr)$ where $Z$ is
a L\'evy process where the L\'evy measure has tail $\gamma(x)\e^{-\psi(x)}$.

\section{The exponential class of distributions}\label{C1C2}
For $F\in \GMDA(e)$ in the previous sections we have discussed the case that $e(x)=1/\lambda(x)$ with
	$$ \lim_{x\to \infty} e(x)= 0.$$
	If $\lim_{x\to \infty} e(x)=\infty$, then $F$ is long-tailed in the sense that
	$ \overline{F}(x- y) \sim \overline{F}(x)$ for any fixed $y$. Convolutions of distributions with long-tailed are well-understood. The intermediate case is that
	$$ \lim_{x\to \infty} e(x)= 1/\gamma, \quad \gamma>0.$$
	For such $F$ we have
	$$ 	\overline  F( x+ s) \sim \e^{-\gamma s} 	\overline  F(x), \quad x\to \infty$$
	for any $s\in \mathbb{R}$, which is also denoted as
	$F \in \mathcal{L}(\gamma)$.
	Note in passing that any distribution $F \in \GMDA(e)$ with upper endpoint infinity satisfies (see e.g. \cite[Prop.\,1.4]{Res})
	\bqn{\label{polin}
		\overline  F(x) \sim  	\overline  H(x)= C \exp\Bigl(- \int_0^x \frac{1}{u(t)} \dd t\Bigr), \quad x\to \infty}%
	for some $C>0$, where $u$ is absolutely continuous with respect to Lebesgue measure, with density $u'$ satisfying  $\lim_{x\to \infty} u'(x)=0$.  Such $H$ is commonly referred to as a von Mises distribution.
	
	It is well-known (\cite{Cline}, \cite{Wata}) that the class of distributions  $\mathcal{L}(\gamma)$ is closed under convolution.  In the particular case that the $X_i$ have tails
	\bqn{ \overline F_i(x) = \ell_i(x)  x^{\gamma_i-1} \e^{-k x^\beta}, \quad 1 \le i\le n,}%
	where $\ell_i$'s are  positive slowly varying functions and $\beta=1, \gamma_i>0, i\le n,k>0$ we have in view of
	Theorem 2.1 in \cite{EHJL} (see also Theorem 6.4 ii) in \cite{APakes})
	\bqn{ \label{jin}
		\pk{S_n> x} \sim  \frac{k^{n-1}}{\Gamma(\gamma_0)}
		x^{\gamma_0-1}\prod_{i=1}^n \ell_i(x) \e^{-k x^\beta }.
	}%
	where $\gamma_0=\sum_{i=1}^n \gamma_i$.
	If \eqref{jin} holds with $\beta>1$, then for non-negative $X_i$'s using the $\beta$-norm argument we have as in Section \ref{S:Bounds}
\bqn{ \label{asm}
	 \pk{S_n> x}  \le \pk{ X_1^{\beta}+ \cdots +X_n^{\beta} > x^\beta /n^{\beta-1}}
}
	for any $x>0$. Since $\mathbb{P}( X_1^{\beta}> x) \sim \ell_i(x^{1/\beta}) x^{(\gamma_i- 1)/\beta} \e^{- k x}$, then by \eqref{jin} and Theorem \ref{Th:19.12a}
	$$\ln \pk{S_n> x} \sim   \ln  \pk{ X_1^{\beta}+ \cdots +X_n^{\beta} > x^\beta /n^{\beta-1}}  \sim  k n (x/n)^{\beta}  $$
and thus the upper bound in \eqref{asm} is logarithmic asymptotically exact.

\section{Applications to Monte Carlo simulation}\label{S:CdMC}

In this section, we write $h_1(x)\approx_{\log}h_2(x)$ if $\log h_1(x)/\log h_2(x)\to 1$
and $ \le_{\log}$ if the $\limsup$ of the ratio of log's is at most 1, and we take the summands 
to have a density like $\gamma(x) \e^{-x^\beta}$ as $x \to \infty$.  

Algorithms for tails $\Prob(S_n > x)$ with large $x$ are one of the traditional objects
of study of the rare-event simulation literature.  
An \emph{estimator} is a r.v.\ $Z(x)$ with $\Exp Z(x)=\Prob(S_n > x)$ and its efficiency
is judged by ratios of the form $r_p(x)=\Exp Z(x)^2/\Prob(S_n > x)^p$. The
estimator will improve upon crude Monte Carlo simulation if $r_1(x)\to 0$ as $x\to\infty$. It is said to have bounded relative error if $r_2(x) $ stays bounded as $x\to\infty$ and to
exhibit logarithmic efficiency if $r_{2-\epsilon}(x)\to 0$ for all $\epsilon>0$ which in turn
will hold if $\Exp Z(x)^2\approx_{\log} \Prob(S_n > x)^2$. These two concepts are
usually considered in some sense optimal. For a survey, see Chapters 
V--VI in~\cite{SSAA}.

The conventional light-tailed rare-event folklore says that a particular kind of importance sampling, \emph{exponential tilting}, is often close to optimal. Here
instead of
$\indi(S_n>x)$ one returns
\[ Z_\theta(x)\  = \
\indi\{S_n > x\} \times L_\theta\quad \text{where\ \ }L_\theta=\widehat{F}[\theta]^n \exp\{ {-}\theta S_n\} \]
where $ X_1,\ldots,X_n $ are i.i.d.\ with density  $f_\theta(y) = \e^{\theta y} f(y) / \widehat{F}(\theta)$
 rather than the given density $f(x)$, and $\theta$ is chosen such that
$\Exp_\theta X =x/n$, that is, $\theta=\lambda(x/n)$. The standard efficiency
results do, however, require both $n\to\infty$ and $x\to\infty$ such that $nx\sim z$
for some $z>\Exp X$ and therefore
do not deal with a fixed $n$, the object of this paper. It is believed that the scheme
is still often close to optimal in this setting, but very few rigorous results in this direction
has been formulated. We give one such in Proposition~\ref{Prop:PL} below.

One problem that arises is how to simulate from $f_\theta$.
Proposition~\ref{Prop:18.12a} tells us that $f_\theta$ is asymptotically normal
with mean $x/n$ and variance $1/\lambda'(x/n)$ when $\theta=\lambda(x/n)$.
So we simulate using acceptance--rejection with a moment-matched
gamma distribution as proposal, and our acceptance ratio will increase to 1 as $x \to \infty$.
To be specific, we take a $\mathsf{Gamma}(a, b)$ proposal, which has a density
$f_{a,b}(y) \propto y^{a-1}\e^{-b y}$, where $a = x^2 \lambda'(x/n)/n^2$, and $b=x \lambda'(x/n)/n$.
The reason we do not directly use a the limiting normal distribution as a proposal is that the tail of the normal distribution is too light when $\beta \in (1,2)$.

\begin{commentaryRemark}
The acceptance ratio can be improved for small $x$ by locally searching for the optimal proposal, that is, the distribution with parameters
\[ (\mu^*,\sigma^*) = \argmin_{\mu,\sigma > 0} \,\, \max_{y \ge 0} \, \frac{ f_{\lambda(x/n)}(y) }{ f_{\mathrm{Prop}}(y ; \mu,\sigma^2) } \,. \]
The asymptotic $(\mu,\sigma)=(x/n, 1/\sqrt{\lambda'(x/n)})$ can be used as the initial search point. In experiments, it seems that the asymptotic variance is close to optimal, whereas some efficiency can be gained by adjusting the mean parameter.
\halmoss\end{commentaryRemark}

\begin{proposition}\label{Prop:PL}
The estimator $Z_\theta(x)$ exhibits logarithmic efficiency.
\end{proposition}
\begin{proof}
We first note that
\begin{align*}
 \Exp_\theta[ Z_\theta(x)^2 ]\ =\ \Exp_\theta[L_\theta^2;\,S_n>x]
 \ =\ \Exp[L_\theta;\,S_n>x]\ \le\ \e^{-\theta x } \widehat F[\theta]^n  \Prob(S_n > x)\,.
\end{align*}
By Corollary~\ref{SACor4.2} and \eqref{18.12a},
\[\Fb^{*n}(x)\ \approx_{\log}\ \exp \{ n(x/n)^\beta \}\,,\ \  \widehat{F}[\lambda(x/n)]^{n}\ \approx_{\log}\ \exp \{ n(\beta-1)(x/n)^\beta \} \,.\]
From $\theta=\lambda(x/n)=\beta (x/n)^{\beta-1}$ we then get
\begin{align*}\frac{ \Var_{\theta}(Z_\theta(x)) }{ \Prob(S_n>x) }\ &\le\ 
\frac{ \Exp_\theta[ Z_\theta(x)^2 ] }{ \Prob(S_n>x) }\\ &\le_{\log}\ 
\exp\bigl\{-\theta x + n(\beta-1)(x/n)^\beta + n (x/n)^{\beta}\bigr\} \\
&=\ \exp\bigl\{-\beta (x/n)^{\beta-1} x + n\beta(x/n)^\beta\bigr\} \ =\  1, 
\end{align*}
completing the proof.
\end{proof}

Some  estimators  based on conditional Monte Carlo ideas
are discussed in \cite{SACdMC} and efficiency properties derived in some
special cases. The algorithms do improve upon crude Monte Carlo, though logarithmic
efficiency is not obtained. The advantage is, however, that they are much easier implemented than the above exponential tilting scheme. The next two propositions extend 
results of \cite{SACdMC} to more general tails. 

\begin{proposition}\label{Prop:SA7.2a} Consider the conditional Monte Carlo
estimator $Z_{\rm Cd}(x)=\Fb(x-S_{n-1})$ of $\Prob(S_n>x)$. Then
$\limsup r_p(x)<\infty$ whenever $p<p_n$ where $p_n=n^{\beta-1}c_n$ with $c_n$ given by
\eqref{SA7.2f} below. Here $p_n>1$.
\end{proposition}
\begin{proof} We have
$\Exp Z_{\rm Cd}(x)^2\,=\, \int \Fb(x-y)^2 f^{*(n-1)}(y)\, \dd y$
where the asymptotics of the integral is covered by Theorem~\ref{Th:9.11a}. In the setting
there, $c_1=2$, $c_2=1/(n-1)^{\beta-1}$ which gives $\theta_1=1/(1+\mu)$,
$\theta_2=\mu/(1+\mu)$ where $\mu=2^{1/(\beta -1)}(n-1)$. The result gives that
$\Exp Z_{\rm Cd}(x)^2\approx_{\log} \e^{-c_nx^\beta}$ where
\begin{equation}\label{SA7.2f}c_n\ = c_1\theta_1^\beta+c_2\theta_2^\beta\ =\ \frac{2+2^{\beta/(\beta-1)}(n-1)}
{\bigl(1+2^{1/(\beta-1)}(n-1)\bigr)^\beta}\,.
\end{equation}
Since $\Prob(S_n>x)\approx_{\log}\e^{-x^\beta/n^{\beta-1}}$, this implies the first assertion of the proposition. To see that
$p_n>1$, note that for $a>1$
\[ n^{\beta-1}\frac{a^{\beta-1}+a^\beta (n-1)}{\bigl(1+a(n-1)\bigr)^\beta}\ =\ \Bigl[\frac{na}{1+a(n-1)}\Bigr]^{\beta-1}\ >\ \Bigl[\frac{na}{na}\Bigr]^{\beta-1}\ =\ 1
\]
and take $a=2^{1/(\beta-1)}$.
\end{proof}

We finally consider the so-called Asmussen--Kroese estimator
\begin{equation}\label{AK}
Z_{\rm AK}(x)\ =\ n\,\Fb\bigl(M_{n-1}\vee(x-S_{n-1})\bigr)\,.
\end{equation}
where $M_{n-1}=\max(X_1,\ldots,X_{n-1})$. It was initially developed in~\cite{AK} with heavy tails
in mind, but it was found empirically in~\cite{SACdMC} that it also provides some
variance reduction for light tails, in fact more than $Z_{\rm Cd}(x)$. We have:\\
\begin{proposition}\label{Prop:SA7.2b} Consider the 
estimator $Z_{\rm AK}(x)$ of $\Prob(S_n>x)$ with $n=2$. Then
$\limsup r_p(x)<\infty$ whenever $p<3/2$. 
\end{proposition}
\begin{proof} When $n=2$, we have $M_{n-1}=S_{n-1}=X_1$ and so the analysis
splits into an $X_1>x/2$ and an $X_1\le 2$ part. The first is
\begin{align*}\MoveEqLeft
 \Exp\bigl[Z_{\rm AK}(x)^2;\,X_1>x/2\bigr]\ =\ 4\int_{x/2}^\infty
 \Fb(y)^2f(y)\,\dd y\\ & \approx_{\log}\ \int_{x/2}^\infty
\e^{-2y^\beta}\e^{-y^\beta}\,\dd y\  \approx_{\log}\ \e^{-3x^\beta/2^\beta}.
\end{align*}

The second part is
\begin{align*}\MoveEqLeft
 \Exp\bigl[Z_{\rm AK}(x)^2;\,X_1\le x/2\bigr]\ =\ 4\int_{-\infty}^{x/2}
 \Fb(x-y)^2f(y)\,\dd y\\ &  \ =\ 4\int_{x/2}^\infty
 \Fb(y)^2f(x-y)\,\dd y\ = \ 4I_1+4I_2
\end{align*}
where $I_1$ is the integral over $[x/2,ax)$ and $I_2$ is the one over $[ax,\infty)$.
Here we take
$a=(3/2)^{1/\beta}/2$; 
since $\beta>1$, we have $a<3/4<1$. Let further $b=a-1/2$. Then
\begin{align*}
I_2\ &=\ \int_{ ax}^\infty  \Fb(y)^2\Oh(1)\dd y\ \approx_{\log} \ \int_{ ax}^\infty  \e^{-2x^\beta}\Oh(1)\dd y\\
 &\approx_{\log}\    \e^{-2a^\beta x^\beta}\ =\  \e^{-3 x^\beta/2^\beta}\,,\\
I_1\ &\approx_{\log}\ \int_{x/2}^{ ax} \exp\bigl\{-2y^\beta-(x-y)^\beta\bigr\}\\& =\ 
\int_0^{ bx} \exp\bigl\{-2(x/2+z)^\beta-(x/2-z)^\beta\bigr\}\,\dd z .
\end{align*}
By convexity of $v\mapsto v^\beta$, we have \[(u+v)^\beta\ =\ u^\beta(1+v/u)^\beta \ \ge u^\beta(1+\beta v/u)\ =\ 
 u^\beta+\beta v  u^{\beta-1}\]
 for $u>0$ and $-u<v<\infty$. Taking $u=x/2$ gives
 \begin{align*}I_2\ \le_{\log}\ \int_0^{ bx} \exp\bigl\{-3x^\beta/2^\beta-\beta z(x/2)^{\beta-1}\bigr\}\,\dd z\ =\
 \e^{-3 x^\beta/2^\beta}\oh(1)\,,
 \end{align*}
 completing the proof.
 \end{proof}

\appendix* \label{S:ProofofTheorem}

For the proof of Theorem~{\ref{Th:19.12a}, we first note that,
as shown in BKR, that as $x\to \infty$
\begin{align}
\label{2.12h} \frac{\lambda'(x)}{{\lambda(x)}^2}\ &\to\ 0.\\
\label{2.12ib} \frac{\gamma'(x)}{{\sqrt{\lambda'(x)}}\gamma(x)}\ &\to\ 0.
\end{align}
In view of Proposition 3.2 in BKR, \eqref{2.12ib} need not hold for $\gamma$ itself but does for  a tail equivalent version, with which $\gamma$ can be replaced w.l.o.g. This implies
\begin{align}
\label{19.12c} \lambda\text{\ is flat for }\psi.
\end{align}
Indeed, given $y$ it holds for some $x^*$ between $0$ and $x+y/\sqrt{\lambda'(x)}$ that
\[\lambda\bigl(x+y/\sqrt{\lambda'(x)}\bigr)\ =\ \lambda(x)+\frac{\lambda'(x^*)}{\sqrt{\lambda'(x)}}y\ =\
\lambda(x)+\Oh\bigl(\sqrt{\lambda'(x)}\bigr)\ = \lambda(x)\bigl(1+\oh(1)\bigr)\]
where the $\Oh(\cdot)$ estimate follows from a known uniformity property of self-neglecting functions and the $\oh(\cdot)$ estimate
by \eqref{2.12h}.

Using further \eqref{2.12h} we have that $e=1/\lambda$ is self-neglecting.

\begin{proofof}{Theorem~{\ref{Th:19.12a}} {\rm (i)}}
Write $\overline{H}(x)=\gamma(x)\e^{-\psi(x)}/\lambda(x)$. 
Then
\begin{align*}\overline{H}'(x)\ &=\ \Bigl[\gamma(x)+\frac{\gamma'(x)}{\psi'(x)}-
\frac{\gamma(x)\psi''(x)}{\psi'(x)^2}\Bigr]\e^{-\psi(x)}\\
&=\ \gamma(x)\Bigl[1+\frac{\gamma'(x)}{\gamma(x)\psi'(x)}-
\frac{\psi''(x)}{\psi'(x)^2}\Bigr]\e^{-\psi(x)}.
\end{align*}
Here the last term in $[\cdot]$ goes to 0 according to \eqref{2.12h}. This
together with \eqref{2.12ib}
also gives
\[\frac{\gamma'(x)}{\gamma(x)\psi'(x)}\ =\ \frac{\gamma'(x){\psi''}^{-1/2}}{\gamma(x)}\cdot
\frac{{\psi''}^{1/2}}{\psi'(x)}\ =\ \oh(1)\cdot\,\oh(1)\ =\ \oh(1).\]
Thus $\overline{H}'(x)\sim f(x)$ which implies $\overline{H}(x)\sim \Fb(x)$.
\end{proofof}

We also have this an alternative proof for part (i).

\begin{proofof}{Theorem~{\ref{Th:19.12a}} {\rm (i)}}
Using integrations by parts yields
\begin{align*}\int_x^\infty f(y)\,\dd y\ &= \int_0^\infty  \frac{\gamma(x+y)}{\psi'(x+y)}
\cdot \psi'(x+y)\e^{-\psi(x+y)} \,\dd y  \\ &=\
 \frac{\gamma(x)}{\psi'(x)}\e^{-\psi(x)}\ -\  \int_0^\infty  \frac{\dd}{\dd y}\Bigl[\frac{\gamma(x+y)}{\psi'(x+y)}\Bigr]
\cdot \e^{-\psi(x+y)} \,\dd y.
\end{align*}
But by the same estimates as in Proof 1, the first part of the integrand is $\oh\bigl(\gamma(x)\bigr)$
so that the whole integral is $\oh\bigl(\Fb(x)\bigr)$.
\end{proofof}

The following lemma is just a reformulation of part (ii) of the theorem, proved in BKR.
\begin{lemma}\label{Lemma:19.12a} For any two pairs $(\gamma_1,\psi_1)$, $(\gamma_2,\psi_2)$
satisfying the assumptions of Section~\ref{S:Intr}, it holds that
\begin{equation}
	\int_{-\infty}^\infty \gamma_1(z)\e^{-\psi_1(z)}\cdot \gamma_2(x-z)\e^{-\psi_2(x-z)}\,\dd z
\end{equation}
has the asymptotics given by Theorem~\ref{Th:19.12a}(ii).
\end{lemma}

\begin{proofof}{Theorem~{\ref{Th:19.12a}} {\rm (ii)}}
This is a reformulation of Theorem 1.1 in BKR. 
Since by \eqref{19.12a} $q_1'+q_2'=1$ we have the claimed relation between $\lambda$ and 
$\lambda_1,\lambda_2$, namely 
\bqn{ \label{remm}
 \lambda(x)\ =\ \lambda_1\bigl(q_1(x)\bigr)q_1'(x)+ \lambda_2\bigl(q_2(x)\bigr)q_2'(x)\ =\ \lambda_1(q_1)\ =\ \lambda_2(q_2)
}
establishing the proof. 
\end{proofof}

\begin{proofof}{Theorem~{\ref{Th:19.12a}} {\rm (iii)}}
We have that $\e^{- \psi_i(x)},i=1,2$ is a von-Mises function (see \eqref{polin}) and thus  $\e^{- \psi_i(x)} \in \GMDA(e_i),i=1,2$ with $e_i=1/\lambda_i$. Since further $e_i$'s are self-neglecting and by \eqref{2.12h} $r_i(x)= \sqrt{\lambda_i(x)}/\lambda_i(x) \to 0$
as $x\to \infty$ we have that 
$$ \lim_{x\to \infty}\frac{\gamma_i(x+ e_i(x)y )}{\gamma_i(x)}=
\lim_{x\to \infty}\frac{\gamma_i(x+ yr_i(x)/ \sqrt{\lambda_i(x)} )}{\gamma_i(x)} =1$$
uniformly on bounded $y$-intervals. Hence  $F_i \in \GMDA(e_i)$. In view of Proposition 3.2 in BKR we can find smooth $\gamma_i^*$'s such that $\overline{H}_i(x)= \gamma_i^*(x)\e^{- \psi_i(x)}/\lambda_i(x)$ is asymptotically equivalent to $\overline{F}_i(x)$ as $x\to \infty$. Since also $H_i\in \GMDA(e_i)$ and $\lim_{x\to \infty} \lambda_i(x)= \infty$, then for any $c>0$ we have
$$ \lim_{x\to \infty} \frac{\overline{H_i}(x+ c)}{\overline H_i(x)}=
0, \quad i=1,2.$$
Consequently, Corollary 1 in \cite{EHM} yields $\overline{H_1*H_2}(x) \sim
\overline{F_1*F_2}(x)$ and thus the claim follows from ii).\\ 	
By the above, we can find the asymptotics of $\overline{F_1*F_2}(x)$ assuming that $F_i$'s possess a density, so alternatively we have
\begin{align}\overline{F_1*F_2}(x)\ &=\ \int_{-\infty}^\infty \gamma_1(z)\e^{-\psi_1(x)}\cdot\frac{\gamma_2(x-z)}
{\lambda_2(x-z)} \e^{-\psi_2(x-z)}\,\dd z
\end{align}
But by \eqref{19.12c}, $\gamma_2/\lambda_2$ is flat for $\psi_2$, so using Lemma~\ref{19.12a}
with $\gamma_2$ replaced by $\gamma_2/\lambda_2$ gives that
this integral asymptotically equals $\gamma(x)\e^{-\psi(x)}/\gamma_2\bigl(q_2(x)\bigr)$.
But in view of \eqref{remm} this is the same as $\gamma(x)\e^{-\psi(x)}/\lambda(x)$. 
This completes the proof.
\end{proofof}

\acknowledgements{PJL was supported by an Australian Government Research Training Program Scholarship and an Australian Research Council Centre of Excellence for Mathematical \& Statistical Frontiers Scholarship.}


\begin{thebibliography}{99}
	\bibitem{APakes} \n{A.A. Adler \and G. Pakes},
	{\it On relative stability and weighted laws of large numbers,}
	Extremes, 20 (2017), pp. 1--31. 
	
	\bibitem {SACdMC} \n{S. Asmussen},
	{\it Conditional Monte Carlo for sums, with applications to insurance and finance,}
	Submitted (2017), available from {\tt thiele.au.dk/publications}.

	\bibitem{RP} \n{S. Asmussen \and H. Albrecher},
	{\it Ruin Probabilities, 2nd ed.,}
	World Scientific, 2010.

	\bibitem{SSAA} \n{S. Asmussen \and P.W. Glynn},
	{\it Stochastic Simulation, Algorithms and Analysis,}
	Springer-Verlag, 2007.

	\bibitem{AK} \n{S. Asmussen \and D.P. Kroese}, 
	{\it Improved algorithms for rare event simulation with heavy tails,}
	Adv.\ Appl.\ Probab.\ 38 (2006), pp. 545--558.

	\bibitem{BKR03} \n{A.A. Balkema, C. Kl\"uppelberg \and S.I. Resnick},
	{\it Domains of attraction for exponential families,}
	Stoch.\ Proc.\ Appl.\ 107 (2003), pp. 83--103.

	\bibitem{BKR} \n{A.A. Balkema, C. Kl\"uppelberg \and S.I. Resnick},
	{\it Densities with Gaussian tails,}
	Proc.\ London Math.\ Soc.\ 66 (1993), pp. 568--588.

	\bibitem{BE} \n{A.A. Balkema \and P. Embrechts},
	{\it High Risk Scenarios and Extremes,}
	European Mathematical Society, 2007.

	\bibitem{OBN} \n{O.E. Barndorff-Nielsen \and C. Kl\"uppelberg},
	{\it A note on the tail accuracy of the saddlepoint approximation,}
	Annales de la facult\'e des sciences de Toulouse 6$^e$ s\'erie 1 (1992), pp. 5--14.

	\bibitem{Cline} \n{D.B.H. Cline}, 
	{\it Convolution tails, product tails and domains of attraction,}
	Probab.\ Th.\ Rel.\ Fields 72 (1986), pp. 529--557.

	\bibitem{EKM} \n{P. Embrechts, C. Kl\"uppelberg \and T. Mikosch},
	{\it Modeling Extreme Events for Insurance and Finance,}
	Springer, 1997.

	\bibitem{EHM} \n{P. Embrechts, E. Hashorva \and T. Mikosch},
	{\it Aggregation of log-linear risks,}
	Journal of Applied Probability 51A (2014), pp. 203--212.

	\bibitem{EHJL} \n{E. Hashorva \and J. Li},
	{\it Asymptotics for a discrete-time risk model with the emphasis on financial risk,}
	Probab.\  Engineer.\ Informat.\ Sci.\ 28 (2014), pp. 573--588.

	\bibitem{JLJ} \n{J.L. Jensen},
	{\it Saddlepoint Approximations,}
	Clarendon Press, Oxford, 1995.

	\bibitem{Res} \n{S.I. Resnick},
	{\it Extreme Values, Regular Variation and Point Processes,}
	Springer-Verlag, 2008.

	\bibitem{Rootzen} \n{H. Rootz\'en},
	{\it A ratio limit theorem for the tails of weighted sums,}
	Ann.\ Probab. 15 (1987), pp. 728--747.

	\bibitem{Wata} \n{T. Watanabe},
	{\it Convolution equivalence and distributions of random sums,}
	Probab.\ Theory Relat.\ Fields 142 (2008), pp. 367--397.

\end{thebibliography}
\end{document}